\documentclass{amsart}%\usepackage[all]{xypic}
\usepackage{amsfonts,amssymb,amsmath,amsthm,mathrsfs,amscd}
\usepackage{url}
\usepackage{enumerate}
\begin{document}
\urlstyle{sf}
\newtheorem{thrm}{Theorem}\numberwithin{thrm}{section}
\newtheorem{lem}[thrm]{Lemma}
\newtheorem{prop}[thrm]{Proposition}
\newtheorem{cor}[thrm]{Corollary}
\newtheorem{conj}[thrm]{Conjecture}
\theoremstyle{definition}
\newtheorem{definition}[thrm]{Definition}
\newtheorem{rem}[thrm]{Remark}
\newtheorem{ex}[thrm]{Example}
\numberwithin{equation}{section}
\newfont{\cyrm}{wncysc10}\newcommand{\sha}{\text{\cyrm{X}}}\

\title[Rational points on curves over function fields]{Rational points on curves over function fields}
\author{Am\'{\i}lcar Pacheco}
\address{Universidade Federal do Rio de Janeiro, Instituto de Matem\'atica, Brasil, Av. Athos da Silveira Ramos 149, CT, Bl. C, Cidade Universit\'aria, Ilha do Fund\~ao, Caixa Postal 68530, 21941-909 Rio de Janeiro, RJ, Brasil}
\email{amilcar@acd.ufrj.br}
\author{Fabien Pazuki}
\address{Institut de Math\'ematiques de Bordeaux, Universit\'e Bordeaux I,
351, cours de la Lib\'eration, 
33405 Talence, France}
\email{fabien.pazuki@math.u-bordeaux1.fr}
\thanks{This paper was completed during the sabbatical leave of the first author which was spent at the Institut de Math\'ematiques de Jussieu, Paris. He would like to thank this institution for its support. This stay had also the support of  CNPq (Brazil), CNRS (France) and Paris Science Foundation. The first author would also like to thank UFRJ for giving him this leave. Both authors would like to thank Universit\'e Bordeaux 1, ANR HAMOT and ANR ARIVAF (France).}
\date{\today}
\begin{abstract}
We provide in this paper an upper bound for the number of rational points on a non-isotrivial curve defined over a one variable function field over a finite field. The bound only depends on the curve and the field, and not on the Jacobian variety of the curve.
\end{abstract}
\maketitle

\section{Introduction}

Let $k$ be a finite field of cardinality $q$ and of positive characteristic $p$. Let $\mathcal{C}$ be a smooth, projective, geometrically connected curve defined over $k$ of genus $g$. Denote by $K=k(\mathcal{C})$ its function field. Let $K_s$ be a separable closure of $K$. Given a smooth, projective, geometrically connected curve $X$ defined over $K$ of genus $d\ge2$, the analogue of the {Mordell}'s conjecture asks whether the set $X(K)$ is finite.

This does not come without a constraint, otherwise this question would have a trivial negative answer. One has to assume that $X$ is non-isotrivial. This means that there does not exist a smooth projective geometrically connected curve $X_0$ defined over a finite extension $l$ of $k$ and a common extension $L$ of both $K$ and $l$ such that $X\times_KL\cong X_0\times_lL$ (\emph{cf.} \cite{Sam}).  Under the aforementioned condition the finiteness of $X(K)$ is a theorem due to {Samuel} \cite{Sam}. 

Our purpose is to give an effective upper bound for the cardinality of the set $X(K)$ in terms of the minimal number of invariants associated with our given geometric situation. Namely, our upper bound will depend on the following parameters. (i) The genus $d$ of $X/K$; (ii) the genus $g$ of $\mathcal{C}/k$; (iii) the inseparable degree $p\sp e$ of the map $\mathcal{C}\to\bar{\mathscr{M}}_g$ from $\mathcal{C}$ to the compactification of the moduli scheme of genus $g$ curves associated to a model $\mathcal{X}\to\mathcal{C}$ of $X/K$  \footnote{Observe that $p\sp e$ does not depend on the choice of the model $\mathcal{X}\to\mathcal{C}$, for a further discussion see Section \ref{secdescent}.}.  (iv) The conductor $f_{X/K}$ of $X/K$ (this will be defined later in the text). Let us insist on the fact that this bound does not depend on the Jacobian variety $J_X$ of the curve $X$. The rank $r$ of the Mordell-Weil group $J_X(K)$ is not used on the bound. In the geometric case this rank is bounded in terms of $d$ and $g$ (cf. Ogg's bound, see Remark \ref{remrank}). We observe, however, that  the bound in terms of  the conductor of $J_X/K$ would be stronger (cf. Proposition \ref{propf}), but the point is to show that the bound can be expressed in terms of only the curve itself. Our main result is the following theorem.
\vspace{1cm}

\begin{thrm}\label{bound}
Let $k$ be a finite field of cardinality $q$ and characteristic $p$, $\mathcal{C}$ a smooth, projective, geometrically connected curve defined over $k$ of genus $g$ and denote by $K=k(\mathcal{C})$ its function field.  Let $X/K$ be a smooth, projective, geometrically connected curve defined over $K$ of genus $d\ge2$. We suppose that $X$ is non-isotrivial.  
\begin{enumerate}
\item[(a)] If $X$ is defined over $K$, but not over $K^p$, then the following inequality holds: 
$$\#X(K)\leq \,p^{2d\cdot(2g+1)+f_{X/K}} \cdot3^d \cdot(8d-2)\cdot d!.$$
Denote the right hand side of the latter inequality by $C_{\mathrm{BV}}$. 
\item[(b)] More generally, suppose that $p>2d+1$. If $X$ is defined over $K^{p^e}$, but not over $K^{p^{e+1}}$ for some natural integer $e$,  then
$$\#X(K)\leq  C_{\mathrm{BV}}\cdot  C_{\mathrm{desc}}\sp e,$$
where one can take $C_{\mathrm{desc}}=q\sp{c_0}$ and $c_0=g-1+f_{X/K}+\frac{1}{2}\cdot p^{e+1} \cdot d\cdot (2g-2+2^{4d^2}\cdot f_{X/K})$. 
\end{enumerate}
\end{thrm}

\begin{rem}
In a recent paper \cite{CoUlVo},  Concei\c c\~ao, Ulmer and Voloch provide some explicit examples of curves $X_a$ for which the number of rational points cannot be bounded by a quantity independent of $X_a$. Consider the curve $X_a/\mathbb{F}_p(t)$ defined by the affine equation  $y^2=x\cdot(x^r+1)\cdot(x^r+a^r)$, where $p>3$ and $r$ is coprime to $2p$ and $a=t^{p^{n}+1}$. Say $n=2^m$ with $m\in{\mathbb{N}}$ big enough. Then $\#X_a(\mathbb{F}_p(t))\geq d(n)\gg \log n\gg \log\log h(X_a/\mathbb{F}_p(t))\gg \log\log f_{J_{X_a}/\mathbb{F}_p(t)}$, where the last step is obtained thanks to \cite[Corollary 6.12]{HiPa}. In fact, the height $h(X_a/\mathbb{F}_p(t))$ is the height of the equation defining the curve which is comparable to the theta height of its Jacobian variety which is comparable with the differential height of the Jacobian variety. The inequality mentioned above relates the conductor of the Jacobian with its differential height.
\end{rem}

The history of explicit upper bounds for $\#X(K)$ starts with the work of {Szpiro} \cite{Szp} which in fact gives an explicit upper bound for the height of points in $X(K)$. This depends, however, on the geometry of a semi-stable fibration on curves $\phi:\mathcal{X}\to\mathcal{C}$ which gives a minimal model of $X/K$ over $\mathcal{C}$. One of the goals of the current paper is to obtain a bound which does not depend on the geometry of any model of $X/K$ over $\mathcal{C}$. 

We start with an upper bound for the number of elements of $X(K)$, when $X$ is defined over $K$, but not over $K\sp p$. This follows from a result due to {Buium} and {Voloch} \cite{BuVo}. In fact, their result gives an explicit proof of a conjecture of Lang, Mordell's conjecture is a particular case of the latter. We then extend the first result to curves which can be defined over $K\sp{p\sp n}$ for some integer $n\ge1$. The crucial step is the $F$-descent of abelian varieties in characteristic $p>0$ (see Section \ref{secdescent}). 

\section{Proof of Theorem \ref{bound} part (a)}
We start by recalling:

\begin{thrm}[{Buium-Voloch}] \cite[Theorem]{BuVo}\label{bvthm}
Let $k$ be a finite field of characteristic $p$, $K$ a one variable function field over $k$, $X/K$ a smooth, projective, geometrically curve defined over $K$ of genus $d\ge2$. We suppose that $X$ is not defined over $K\sp p$. Let $\Gamma$ a subgroup of $J_X(K_s)$ such that $\Gamma/p\Gamma$ is finite. The following inequality holds : 
$$\#(X\cap\Gamma)\leq\#(\Gamma/p\Gamma) \cdot p^d \cdot3^d \cdot(8d-2) \cdot d!.$$
\end{thrm}

\begin{rem}\label{remJ}
Take $\Gamma=J_X(K)$. Then by the {Mordell-Weil} theorem $J_X(K)/$ \linebreak $pJ_X(K)$ is a finite group. Writing $J_X(K)=\mathbb{Z}^r\times J_X(K)_{\mathrm{tor}}$, where $r=\mathrm{rk}\,J_X(K)$, one has $J_X(K)/pJ_X(K)=(\mathbb{Z}/p\mathbb{Z})^r\times J_X(K)_{\mathrm{tor}}/pJ_X(K)_{\mathrm{tor}}$. Its order is bounded from above by $p\sp{d+r}$. Next we discuss an upper bound for the rank.
\end{rem}

\begin{rem}\label{remrank}
Let $k$ be any field and $\mathcal{C}$ smooth projective geometrically connected curve over $k$. Denote by $K=k(\mathcal{C})$ its function field. Let $A/K$ be a non-constant abelian variety over $K$ and denote by $(\tau,B)$ its $K/k$-trace ({cf.} \cite{Lang}). Let $\bar{k}$ be an algebraic closure of $k$. A theorem due to {Lang} and {N\'eron} (\cite{Lang}, \cite{LaNe}) states that the quotient group $A(\bar{k}(\mathcal{C}))/\tau B(\bar{k})$ is a finitely generated abelian group. {A fortiori}, the quotient group $A(K)/\tau B(k)$ is also finitely generated. {Ogg} in the 60's ({cf.} \cite{Ogg1}) produced the following upper bound for the rank of the geometric quotient $A(\bar{k}(\mathcal{C}))/\tau B(\bar{k})$ (hence of $A(K)/\tau B(k)$). Below we define the conductor $f_{A/K}$ of $A/K$. Let $d_0=\dim B$. Then the upper bound is $2d\cdot(2g-2)+f_{A/K}+4d_0\le4d\cdot g+f_{A/K}$. In particular, if $K$ is a one variable function field over a finite field, then $\mathrm{rk}\,A(K)\le4d\cdot g+f_{A/K}$. 
\end{rem}

\begin{definition}
Let $\ell\ne p$ be a prime number. Denote by $T_{\ell}(A)$ the $\ell$-adic {Tate} module of $A$ and define $V_{\ell}(A)=T_{\ell}(A)\otimes_{\mathbb{Z}_{\ell}}\mathbb{Q}_{\ell}$. For each place $v$ of $K$, denote by $I_v$ an inertia group at $v$ (well-defined up to conjugation). Let $\epsilon_v$ be the codimension of the subgroup of $I_v$-invariants $V_{\ell}(A)\sp{I_v}$ in $V_{\ell}(A)$. Let $\delta_v$ be the {Swan} conductor of $H\sp1_{\text{\'et}}(A_{K_s},\mathbb{Q}_{\ell})$ ({cf.} \cite{Serre}). Define the conductor divisor $\mathfrak{F}_{A/K}=\sum_v(\epsilon_v+\delta_v)\cdot [v]$, where $v$ runs through the places of $K$. Denote $f_{A/K}=\deg\mathfrak{F}_{A/K}$.
\end{definition}

\begin{definition}
A {model} of $X/K$ over $\mathcal{C}$ is a smooth projective geometrically connected surface $\mathcal{X}$ defined over $k$ and a proper flat morphism $\phi:\mathcal{X}\to\mathcal{C}$ such that the generic fiber $\mathcal{X}_{\eta}$ of $\phi$ is isomorphic to $X$. Each place $v$ of $K$ is identified with a point of $\mathcal{C}$. Denote by $\kappa_v$ the residue field at $v$ (which is a finite field) and let $\bar{\kappa}_v$ be an algebraic closure of $\kappa_v$. Denote by $\mathcal{X}_v$ the fiber of $\phi$ at $v$. For an algebraic variety $Z$ defined over a field $l$ and for an extension $L$ of $l$, denote $Z_L=Z\times_lL$.
\end{definition}
  
\subsection{Tools from \'etale cohomology} 
\begin{definition}
Let $Z$ be a smooth variety defined over a field $l$ with algebraic closure $\bar{l}$. Denote by $n=\dim Z$, for each $0\le i\le 2n$, let $H\sp i_{\text{\'et}}(X_{\bar{l}},\mathbb{Q}_{\ell})$ be the $i$-th \'etale cohomology group of $Z/l$. Define the Euler-Poincar\'e characteristic of $Z/l$ by $\chi(Z/l)=\sum_{i=0}\sp{2n}(-1)\sp i\dim_{\mathbb{Q}_{\ell}} H\sp i_{\text{\'et}}(Z_{\bar{l}},\mathbb{Q}_{\ell})$. This number is indeed independent from the choice of $\ell$.
\end{definition}

\begin{definition}\label{defartin}
Fix a place $v$ of $K$. The {Artin} conductor of the curve $X$ over $K$ at $v$ is defined as $f_{X/K,v}=-\chi(X_{K_s})+\chi(\mathcal{X}_{v,\bar{\kappa}_v})+\delta_v$, where $\chi(X_{K_s})$, respectively $\chi(\mathcal{X}_{v,\bar{\kappa}_v})$ denotes the {Euler-Poincar\'e} characteristic of $X_{K_s}$, respectively $\mathcal{X}_{v,\bar{\kappa}_v}$. The term $\delta_v$ denotes the {Swan} conductor of $H\sp1(X_{K_s},\mathbb{Q}_{\ell})$ at $v$ ({cf.} \cite[end of p. 414]{LiSa} for the definition of the Artin conductor, \cite{Serre} for the definition of the {Swan} conductor, as well as \cite[\S1]{Blo}).  Define the global conductor of the curve $X/K$ by $f_{X/K}=\sum_vf_{X/K,v}\cdot\deg v$, where $v$ runs through the places of $K$.
\end{definition}

The following proposition is a consequence of the subsequent lemma in \cite{Blo}.

\begin{prop}\label{propf}
We have the inequality $f_{J_X/K}\le f_{X/K}$. 
\end{prop}

\begin{lem}\label{lemBlo}\cite[Lemma 1.2]{Blo}
Fix a place $v$ of $K$ and let $I_v$ be an inertia subgroup of $\mathrm{Gal}(K_s/K)$ at $v$. Then : 
\begin{enumerate}
\item[(I)] $H\sp i_{\text{\'et}}(X_{K_s},\mathbb{Q}_{\ell})\sp{I_v}\cong H\sp i_{\text{\'et}}(\mathcal{X}_{v,\bar{\kappa}_v},\mathbb{Q}_{\ell})$ for $i=0,1$.
\item[(II)] Let $M_v$ be the free abelian group generated by the irreducible components of $\mathcal{X}_{v,\bar{\kappa}_v}$. Since the individual components are not necessarily defined over $\kappa_v$, there is an action of $\hat{\mathbb{Z}}\cong\mathrm{Gal}(\bar{\kappa}_v/\kappa_v)$ on $M_v$. Moreover, there is an exact sequence of $\hat{\mathbb{Z}}$-modules : 
$$0\to\mathbb{Q}_{\ell}(-1)\to M_v\otimes\mathbb{Q}_{\ell}(-1)\to H\sp2_{\text{\'et}}(\mathcal{X}_{v,\bar{\kappa}_v},\mathbb{Q}_{\ell})\to H\sp2_{\text{\'et}}(X_{K_s},\mathbb{Q}_{\ell})\sp{I_v}\to0.$$ 
\end{enumerate}
\end{lem}

\begin{rem}\label{remdiv}
The definition of the conductor given in \cite{LiSa} agrees with that given in \cite{Blo} (up to sign).
\end{rem}

\begin{proof}[Proof of Proposition \ref{propf}]
It follows from the definition of $f_{X/K,v}$, Lemma \ref{lemBlo} and the fact that the action of the {Galois} group $\mathrm{Gal}(K_s/K)$ on the \'etale cohomology groups $H\sp i_{\text{\'et}}(X_{K_s},\mathbb{Q}_{\ell})$ (for $i=0,2$) is trivial that we have an equality : 
$$
f_{X/K,v}=\dim_{\mathbb{Q}_{\ell}}H\sp1_{\text{\'et}}(X_{K_s},\mathbb{Q}_{\ell})-\dim_{\mathbb{Q}_{\ell}}H\sp1_{\text{\'et}}(X_{K_s},\mathbb{Q}_{\ell})\sp{I_v}+m_v-1+\delta_v,
$$
where $m_v$ denotes the number of the irreducible components of $\mathcal{X}_{v,\bar{\kappa}_v}$. The proposition now follows from observing that $H\sp1_{\text{\'et}}(X_{K_s},\mathbb{Q}_{\ell})\cong H\sp1_{\text{\'et}}(J_{K_s},\mathbb{Q}_{\ell})$ (\emph{cf.} \cite[Corollary 9.6]{Milne}). 
\end{proof}

\begin{definition}
Let $l$ be a field of characteristic $p>0$ and $Z/l$ a smooth algebraic variety. Let $F_{\mathrm{abs}}:l\to l$ be the absolute Frobenius map defined by $a\mapsto a\sp p$. We define the smooth variety $Z\sp{(p)}$ by the commutative diagram
$$\begin{CD}
Z\sp{(p)}@>>>Z\\ @VVV@VVV\\ \mathrm{Spec}\,l @>>{F_{\mathrm{abs}}}>\mathrm{Spec}\,l.
\end{CD}$$
The relative Frobenius morphism $F:Z\to Z\sp{(p)}$ is defined so that composed with the upper horizontal arrow of the diagram gives the absolute Frobenius morphism $F_{\mathrm{abs}}:Z\to Z$.  This situation can be iterated by taking for any integer $e\ge1$ to get the $e$-th power $F\sp e:Z\to Z\sp{(p\sp e)}$ of $F$. 
\end{definition}

\begin{proof}[Proof of Theorem \ref{bound} part (a)]
Let $\jmath:X\hookrightarrow J_X$ be the embedding of $X$ into its Jacobian variety. Denote by $X(K)=\{x_1,\cdots,x_m\}$ the finite set of $K$-rational points of $X$. Let $\Gamma$ be the subgroup of $J_X(K)$ generated by the images $\{\jmath(x_1),\cdots,\jmath(x_m)\}$ of these points under the embedding $\jmath$. Observe that 
$$\#(\Gamma/p\Gamma)\le\#(J_X(K)/pJ_X(K))\le p\sp{r+d}\le p\sp{d\cdot(4g+1)+f_{X/K}}$$ 
by Remarks \ref{remJ} and \ref{remrank} and Proposition \ref{propf}. The result is now a consequence of Theorem \ref{bvthm}.
\end{proof}

\section{Proof of Theorem \ref{bound} part (b) : $F$-descent in characteristic $p$}\label{secdescent}
Let $K$ be a one variable function field over a finite field of characteristic $p>0$.

\subsection{Selmer groups}\cite[\S1]{ul}
We start with the more general set-up of an isogeny $f:A\to B$ of non-constant abelian varieties defined  over $K$. We use the convention that all cohomology groups will be computed in terms of the flat site. As a consequence, at the flat site of $K$, we have a short exact sequence of group schemes given by $0\to\ker f\to A\longrightarrow B\to0$. 

For any place $v$ of $K$, let $K_v$ be the completion of $K$ at $v$. Denote by $\mathrm{Sel}(K_v,f)$ the image of the coboundary map $\delta_v:B(K_v)\to H\sp1(K_v,\ker f)$. The global Selmer group $\mathrm{Sel}(K,f)$ is defined as the subset of those elements in $H\sp1(K,\ker f)$ whose restriction modulo $v$ is trivial in $\mathrm{Sel}(K_v,f)$ for every place $v$ of $K$. 

Recall that the Tate-Shafarevich group $\sha(A/K)$ is defined as $\ker(H\sp1(K_v,A)\to\prod_vH\sp1(K_v,A))$. The isogeny $f$ induces a map $f:\sha(A/K)\to\sha(B/K)$ whose kernel is denoted by $\sha(A/K)_f$. Then $\mathrm{Sel}(K,f)$ appears in the following exact sequence of groups: $0\to B(K)/f(A(K))\to\mathrm{Sel}(K,f)\to\sha(A/K)_f\to0$. In practise $\mathrm{Sel}(K,f)$ is finite and effectively computable. 

Denote by $\mathcal{O}_v$ the valuation ring of $K_v$. If both $A$ and $B$ have good reduction over $\mathcal{O}_v$, then the restriction map $H\sp1(\mathcal{O}_v,\ker f)\to H\sp1(K_v,\ker f)$ induces an isomorphism $\mathrm{Sel}(K_v,f)\cong H\sp1(\mathcal{O}_v,\ker f)$. If $L/K_v$ is a Galois extension of degree prime to $\deg\,f$, then the inclusion map $H\sp1(K_v,\ker f)\to H\sp1(L,\ker f)$ induces an isomorphism $\mathrm{Sel}(K_v,f)\cong\mathrm{Sel}(L,f)\sp G$. Similarly, if $L/K$ is a finite Galois extension of degree prime to $\deg f$, then $\mathrm{Sel}(K,f)=\mathrm{Sel}(L,f)\sp G$. 

\subsection{Group cohomology}\cite[Chapter VII, Appendix]{serrelf}
Let $G$ be a group and $A$ a $G$-module (which we do not assume to be abelian). We make the convention that $G$ acts on $A$ by the left. In this context we define cohomology groups as follows. First, $H\sp0(G,A)=A\sp G$. A cocycle is a map $G\to A$ given by $s\mapsto a_s$ such that $a_{st}=a_s\cdot s(a_t)$. Given two cocycles $a,b$ we say that they are cohomologous, if there exists $\alpha\in A$ such that $b_s=\alpha\sp{-1}\cdot a_s\cdot s(\alpha)$ for every $s\in G$. The quotient of the group of cocycles by the group of coboundaries is $H\sp1(G,A)$. Suppose that $A$, $B$ and $C$ are $G$-modules such that  $B\cong A\oplus C$, then (by definition) $H\sp1(G,B)\cong H\sp1(G,A)\oplus H\sp1(G,C)$. 

\subsection{$p$-rank and Lie algebras}\cite[Theorem, p. 139]{mum}\label{differential}
In the case of an abelian variety $\mathfrak{A}$ defined over an algebraically closed field $l$, denote by $\mathfrak{A}\sp{\vee}=\mathrm{Pic}\sp0\,\mathfrak{A}$ its dual abelian variety. Then $\mathrm{Lie}\,\mathfrak{A}\sp{\vee}\cong H\sp1(\mathfrak{A},\mathcal{O}_{\mathfrak{A}})$, moreover under this isomorphism the $p$-th power map on $\mathrm{Lie}\,\mathfrak{A}$ goes over the Frobenius map $F$ on $H\sp1(\mathfrak{A},\mathcal{O}_{\mathfrak{A}})$. In particular, $r=p\text{-rk}\,\mathfrak{A}=\dim_{\mathbb{F}_p}\,\mathfrak{A}[p]=\dim_{\mathbb{F}_p}H\sp1(\mathfrak{A},\mathcal{O}_{\mathfrak{A}})\sp F=\dim_{\mathbb{F}_p}\,(\mathrm{Lie}\,\mathfrak{A})_{\mathrm{ss}}$. It is known from $p$-linear algebra that $\ker F\cong\mu_p\sp{\oplus r}$. Therefore $H\sp1(K_v,\ker F)\cong H\sp1(K_v,\mu_p)\sp{\oplus r}\cong(K_v\sp*/K_v\sp{*p})\sp{\oplus r}$. 

We now return to our original abelian variety $A/K$, and denote by $\varphi:\mathcal{A}\to\mathcal{C}$ its N\'eron model over $\mathcal{C}$. Let $e_{\mathcal{A}}:\mathcal{C}\to\mathcal{A}$ be its neutral section. Denote $\omega_{\mathcal{A}/\mathcal{C}}=e_{\mathcal{A}}\sp*\Omega\sp1_{\mathcal{A}/\mathcal{C}}$ and $\tilde{\omega}_{\mathcal{A}/\mathcal{C}}=\wedge\sp d\omega_{\mathcal{A}/\mathcal{C}}$, where $d=\dim\,A$. The degree of $\tilde{\omega}_{\mathcal{A}/\mathcal{C}}$ is defined as the differential height of $A/K$ and denoted by $h_{\mathrm{diff}}(A/K)$. Then $\tilde{\omega}_{\mathcal{A}/\mathcal{C}}$ corresponds to a unique Weil divisor $\mathcal{D}_{\mathcal{A}/\mathcal{C}}$ on $\mathcal{C}$. 

The relative version of the first paragraph of this section states that if $\varphi\sp{\vee}:\mathcal{A}\sp{\vee}\to\mathcal{C}$ is the dual group scheme of $\varphi:\mathcal{A}\to\mathcal{C}$, then Lie$\,\mathcal{A}\sp{\vee}\cong R\sp1\varphi_*\mathcal{O}_{\mathcal{A}}$. The latter is dual to $\Omega\sp1_{\mathcal{C}}(\mathcal{D}_{\mathcal{A}/\mathcal{C}})$. 

Denote by $\mathscr{C}$ the Cartier operator acting on $\Omega\sp1_{\mathcal{C}}$ (cf. \cite{serremex}). By the previous result and Serre's duality, the $p$-th power map on Lie$\,\mathcal{A}\sp{\vee}$ goes over $F$ on $R\sp1\varphi_*\mathcal{O}_{\mathcal{A}}$ which goes over to $\mathscr{C}$ on $\Omega\sp1_{\mathcal{C}}(\mathcal{D}_{\mathcal{A}/\mathcal{C}})$. In particular, $p$-Lie$\,\mathcal{A}\sp{\vee}$ is dual to $\Omega\sp1_{\mathcal{C}}(\mathcal{D}_{\mathcal{A}/\mathcal{C}})\sp{\mathscr{C}}$. 
%However, it follows from \cite{ra} that $\deg\mathcal{D}_{\mathcal{A}\sp{\vee}/\mathcal{C}}=\deg\mathcal{D}_{\mathcal{A}/\mathcal{C}}= h_{\mathrm{diff}}(A/K)$. 

Let $D_1,\cdots,D_r$ be a basis of $H\sp0(\mathcal{C},p$-Lie$\,\mathcal{A}\sp{\vee})$, then there exists $a_i\in\bar{k}$ such that $D_i\sp p=a_i\cdot D_i$, where $\bar{k}$ denotes the algebraic closure of $k$. Denote $\mathcal{L}_{\mathcal{A}}=\,$Lie$\,\mathcal{A}\sp{\vee}$. In this case the Oort-Tate classification of finite flat group schemes of order $p$ in characteristic $p$ implies that we associate to $a_i$ the group scheme $G_{\mathcal{L}_{\mathcal{A}},0,a_i}=G_{0,a_i}$ over $\mathcal{C}$ (cf. \cite[Chapter III, 0.9]{mil}, \cite{taoo}).

\subsection{Local computation}
\subsubsection{Potential good reduction}
We fix a place $v$ of $K$. Let $K_s$ be a separable closure of $K$ and denote by $I_v$ a inertia subgroup of $\mathrm{Gal}(K_s/K)$ at $v$ (this is well-defined up to conjugation). By definition $A_{K_v}=A\times_KK_v$ has potential good reduction at $v$, if there exists a finite extension $K'$ of $K_v$ such that $A_{K'}=A_{K_v}\times_{K_v}K'$ has good reduction at $v$. By \cite[Theorem 2]{seta}, if $\ell\ne p$ is a prime number and $\rho_{\ell}:\mathrm{Gal}(K_s/K)\to\mathrm{Aut}(T_{\ell}(A))$ is the Galois representation on the Tate module, then $A$ has potential good reduction at $v$ if and only if $\rho_{\ell}(I_v)$ is finite. 

\subsubsection{Description of Selmer groups}\cite[\S3]{ul}
Suppose that we are on this case and let $K'$ be as above. Denote by $v'$ the valuation of $K'$ over $v$. Let $n=-v'(\mathcal{D}_{\mathcal{A}/\mathcal{C}})$. Define $U\sp{[i]}=\{\bar{f}\in K_v\sp*/K_v\sp{*p}\,|\,\mathrm{ord}_v(f)\ge1-i\}$. Apply \cite[Chapter III, \S7.5]{mil} to get $H\sp1(\mathcal{O}_{K'},G_{0,a_i})\cong U_{K'}\sp{[pn]}$. The previous properties of Selmer groups give 
$$\mathrm{Sel}(K',F)=H\sp1(\mathcal{O}_{K'},\ker f)\cong H\sp1(\mathcal{O}_{K'},\mu_p)\sp{\oplus r}\cong(U_{K'}\sp{[pn]})\sp{\oplus r}.$$ 
Then Galois invariants as in Subsection 3.1, we get 
$$\mathrm{Sel}(K,F)\cong\mathrm{Sel}(K',F)\sp{\mathrm{Gal}(K'/K)}\cong(U_{K_v}\sp{[i]})\sp{\oplus r},$$
where $i=-p\cdot v(\mathcal{D}_{\mathcal{A}/\mathcal{C}})$. 

\subsubsection{Potential semi-abelian reduction}
We suppose that $p>2d+1$, where $d=\dim A$. In this case, $A$ acquires everywhere semi-stable reduction over $L=K(A[\ell])$ for any prime $\ell\ne p$ (cf. \cite{groth}). In particular, for the places where the reduction is already good, we are reduced to the latter subsubsection. So we suppose that we are in the case of bad semi-abelian reduction. In this case by \cite{bolura} there exists a semi-abelian variety $G\in\mathrm{Ext}\sp1(B,\mathbb{G}_m\sp t)$ defined over $L_w$, where $B$ is an abelian variety with good reduction at $w$, and a lattice $\Lambda\subset G(L_w)$ such that $A(L_w)\cong G(L_w)/\Lambda$. 

The action of the absolute Fronenius map $F$ of $G$ engenders the semi-abelian variety $G\sp{(p)}\in\mathrm{Ext}\sp1(B\sp{(p)},\mathbb{G}_m\sp t)$, where $B\sp{(p)}$ is the image of $B$ under $F$. One checks that $A\sp{(p)}(L_w)\cong G\sp{(p)}(L_w)/\Lambda\sp{(p)}$, where the lattice $\Lambda\sp{(p)}$ is generated by the vectors obtained from the generators of $\Lambda$ by raising each component to $p$. Recall that there exists an isogeny $V:A\sp{(p)}\to A$ (called the Verschiebung) such that $V\circ F=[p]_A$ and $F\circ V=[p]_{A\sp{(p)}}$. 

The coboundary map is given by $A(L_w)\to H\sp1(L_w,\ker F)$. We have already shown that the latter is isomorphic to $(L_w\sp*/L_w\sp{*p})\sp{\oplus r}$. The previous parametrization composed with $V$ gives then a surjective map $G\sp{(p)}(L_w)/\Lambda\sp{(p)}\twoheadrightarrow(L_w\sp*/L_w\sp{*p})\sp{\oplus r}$. In particular, this implies that the coboundary map is surjective, i.e., $\mathrm{Sel}(L_w,F)\cong(L_w\sp*/L_w\sp{*p})\sp{\oplus r}$. Finally once more taking Galois invariants we get $$\mathrm{Sel}(K_v,F)\cong\mathrm{Sel}(L_w,F)\sp{\mathrm{Gal}(L_w/K_v)}\cong(K_v\sp*/K_v\sp{*p})\sp{\oplus r},$$
(cf. \cite[\S3]{ul}).

\subsection{Global result}
We denote by $v,\mathrm{good}$ the set of places $v$ of $K$ where $A$ has good reduction. Similarly $v,\mathrm{bad}$ denotes the set of places $v$ of $K$ where $A$ has bad reduction. Let $$D=\sum_{v,\mathrm{bad}}[v]-\sum_{v,\mathrm{good}}i_v\cdot[v]\in\mathrm{Div}\,\mathcal{C},\text{ where }i_v=-p\cdot v(\mathcal{D}_{\mathcal{A}/\mathcal{C}}).$$ 
We observe that 
$$0<\deg D\le f_{A/K}+p\cdot h_{\mathrm{diff}}(A/K).$$ 
Note that there exists an injective map $K\sp*/K\sp{*p}\hookrightarrow\Omega\sp1_K$ given by $\bar{f}\mapsto df/f$. Observe that the image is exactly $(\Omega\sp1_{\mathcal{C}})\sp{\mathscr{C}}$. The local results imply :  $\bar{f}\in\mathrm{Sel}(K,F)$ if and only if $df/f\in H\sp0(\mathcal{C},\Omega\sp1_{\mathcal{C}}(-D))\sp{\mathscr{C}}$. By the Riemann-Roch theorem, one gets
\begin{multline*}
\dim_{\mathbb{F}_q}H\sp0(\mathcal{C},\Omega\sp1_{\mathcal{C}}(-D))\sp{\mathscr{C}}\leq \dim_{\mathbb{F}_q}H\sp0(\mathcal{C},\Omega\sp1_{\mathcal{C}}(-D))\\ 
= g-1+\deg D\le g-1+f_{A/K}+p\cdot h_{\mathrm{diff}}(A/K).
\end{multline*}

\begin{rem}
As we have mentioned before $p\sp e$ is the inseparable degree of the map $\mathcal{C}\to\bar{\mathscr{M}}_g$ from $\mathcal{C}$ to the compactified moduli space $\bar{\mathscr{M}}_g$ of genus $g$ curves induced from a model $\mathcal{X}\to\mathcal{C}$ of $X/K$. This invariant is indeed birational. It may be interpreted as follows : $p\sp e$ is the largest power of $p$ such that $X$ is defined over $K\sp{p\sp e}$, but not over $K\sp{p\sp{e+1}}$.
\end{rem}

We need to assume from now on that $p>2d+1$. In this case, if $\ell\ne p$ is a prime number and $L=K(A[\ell])$ , then $A$ has semi-abelian reduction over $L$. Furthermore, since $L/K$ is tamely ramified of degree prime to $p$, then the Swan conductor makes no contribution to $f_{A/K}$. In fact, $\mathfrak{F}_{A/K}=\sum_v\epsilon_v\cdot[v]$. We can now apply the abc-theorem for semi-abelian schemes in characteristic $p>0$ (cf. \cite[Theorem 5.3]{HiPa}) to get the following bound :
$$h_{\mathrm{diff}}(A_L/L)\leq \frac{1}{2}\cdot p^e \cdot d\cdot (2g-2+f_{A_L/L}).$$
By  \cite[Proposition 3.7]{Pach} page 371, one has $f_{A_L/L}\leq [L:K]\cdot f_{A/K}\leq \ell^{4d^2}\cdot f_{A/K}$. Choosing $\ell=2$ (remember that $p>2d+1\geq 3$, so $p\neq 2$) provides:
\begin{equation}
h_{\mathrm{diff}}(A_L/L)\leq \frac{1}{2}\cdot p^e \cdot d\cdot (2g-2+2^{4d^2}\cdot f_{A/K}).
\end{equation}
Let $c_0=g-1+f_{A/K}+\frac{1}{2}\cdot p^e \cdot d\cdot (2g-2+2^{4d^2}\cdot f_{A/K})$. Then $\#\mathrm{Sel}_{A_L}(L,F)\le q\sp{c_0}$. Recall that since $L/K$ is Galois of order prime to $p$, then $\mathrm{Sel}_{A_L}(L,F)^G=\mathrm{Sel}_A(K,F)$, we conclude that $\# \mathrm{Sel}_A(K,F)\leq q^{c_0}$. Denote $C_{desc}=q\sp{c_0}$.

\subsection{Using $F$-descent and finishing the proof}
The following lemma allows us to conclude the proof of item (b) of Theorem \ref{bound}.

\begin{lem}
Let $X\hookrightarrow J_X$ be a curve over a field $K$ as before embedded into its Jacobian variety $J_X$. Suppose $X$ is defined over $K^{p^e}$ and not over $K^{p^{e+1}}$. Suppose that one has the estimate 
$$\#(J_X(K)/F(J_X(K)))\leq C_{\mathrm{desc}}.$$ 
Then one obtains the upper bound  
$$\#X(K)\leq C_{\mathrm{BV}}\cdot C_{\mathrm{desc}}\sp e,$$
where $C_{\mathrm{BV}}=p^{2d\cdot (2g+1)+f_{X/K}} \cdot3^d \cdot(8d-2)\cdot d!$. 
\end{lem}

\begin{proof}
Suppose that $X$ is defined over $K$ but not over $K\sp p$. Then without any further hypothesis the theorem is proven in part (a). Suppose now that $X$ is defined over $K\sp p$ but not over $K\sp{p\sp2}$. Then there exists a smooth geometrically connected projective curve $X_1$ defined over $K$ but not over $K\sp p$ such that 
$$F:X_1\to X_1\sp{(p)}=X$$ 
is the relative Frobenius morphism of $X_1$. Consider the following decomposition in lateral classes 
$$X(K)=\bigcup_iF(X_1(K))+P_i.$$ 
Under the embedding $\jmath:X\hookrightarrow J_X$ this decomposition is included in the decomposition 
$$\bigcup_iF(J_{X_1}(K))+\jmath(P_i).$$
Note that these classes are not necessarily distinct, however this decomposition is contained in the decomposition 
$$\bigcup_lF(J_{X_1}(K))+\alpha_l,$$
where we are now considering all representatives of $J_X(K)$ modulo $F(J_{X_1}(K))$. As a consequence we get 
$$(X(K):F(X_1(K))\le(J_X(K):F(J_{X_1}(K))\le\#\mathrm{Sel}_{J_X}(K,\ker F)\le C_{\mathrm{desc}},$$ Recall that $F$ is purely inseparable, therefore $\#F(X_1(K))\le C_{\mathrm{BV}}$. Finally we get 
$$\#X(K)\le C_{\mathrm{BV}}\cdot C_{\mathrm{desc}}.$$

Suppose now that $X$ is defined over $K\sp{p\sp2}$ but not over $X\sp{p\sp3}$. As before there exist curves $X_1,X_2$ (with the same description as in the last paragraph) such that 
$$X_2\overset F\longrightarrow X_1=X_2\sp{(p)}\overset F\longrightarrow X=X_1\sp{(p)}=X_2\sp{(p\sp2)}.$$ 
In this case we have got inequalities 
$$\#X_1(K)\le\#X_2(K)\cdot\#(J_{X_1}(K)/F(J_{X_2}(K)),$$
$$\#X(K)\le\#X_1(K)\cdot\#(J_X(K)/F(J_{X_1}(K)).$$ 
Observe that $\#(J_{X_1}(K)/F(J_{X_2}(K))\le\#\mathrm{Sel}_{J_{X_1}}(K,\ker F)$ and the only invariant for the upper bound of the latter term is its conductor. Since $J_X$ and $J_{X_1}$ are $F$-isogeneous, their conductors coincide.  Whence 
$$\#X(K)\le C_{\mathrm{BV}}\cdot C_{\mathrm{desc}}\sp2.$$
An easy induction argument then finishes the proof. 
\end{proof}

\section{Further remarks}
\begin{rem}
We would now like to compare our result with a result similar in nature when we replace the one variable function field $K$ defined over a finite field $k$ by a number field $K$. In order to do this we refer to the work of {R\'emond} ({cf.} \cite{Rem}).
\end{rem}

\begin{thrm}[R\'emond]
Let $X$ be a smooth, projective, geometrically connected curve of genus $d\geq 2$ defined over a number field $K$, then one has
$$\#X(K)\leq (2\sp{38+2d}\cdot[K:\mathbb{Q}]\cdot d\cdot\max(1,h_{\Theta}))\sp{(r+1)\cdot d\sp{20}},$$
where $h_\Theta$ is the theta height of $J_X$ and $r=\mathrm{rk}\,J_X(K)$. 
\end{thrm}

\begin{rem}
Using Proposition 5.1 page 775 of \cite{Rem}, one has $r\ll \log f_{J_X/K}$, as in the function field case, but the bound on the number of points is still dependent on the height of the Jacobian variety. To be more precise, {R\'emond} shows in {loc. cit.} how to produce a bound depending on the height of a model of the curve (and not of its Jacobian variety), but it seems difficult to get rid of this height. It would be a consequence of a conjecture of {Lang} and {Silverman}, as explained in the introduction of \cite{Paz}. Note that in the function field case, the height of the Jacobian variety $J_X$ is comparable to the degree of its conductor $f_{J_X/K}$, as shown in \cite[Corollary 5.12]{HiPa}.
\end{rem}

%\noindent{{2010 Math. subject classification:} 11G35, 11R58, 14G40, 14G25, 14H05.}

%\noindent{{Keywords:} Abelian varieties, rational points, curves, function fields.}


\begin{thebibliography}{AAAAAA}
\bibitem[BoLuRa90]{bolura}S. Bosch, W. L\'uktebhomert, M. Raynaud, \emph{N\'eron Models}, Springer-Verlag (1990).
\bibitem[Bl87]{Blo} S. Bloch,  \emph{De Rham cohomology and conductors of curves}, Duke Math. {\bf 54} No.2 (1987): 295--308.
\bibitem[BuVo96]{BuVo} A. Buium,  and J. F. Voloch,  \emph{Lang's conjecture in characteristic p: an explicit bound}, Compositio Math. {\bf 103} (1996): 1--6.
\bibitem[CoUlVo12]{CoUlVo} R. Concei\c c\~ao, D. Ulmer, J. F. Voloch, \emph{Unboundedness of the number of rational points on curves over function fields}, New York J. Math. {\bf 18} (2012): 291--293.
\bibitem[Gr72]{groth}A. Grothendieck, \emph{Mod\`eles de N\'eron} in S\'em. G\'eom. Alg. 7, exp. IX, 315-623, Lect. Notes in Math. \textbf{288}, (1972).
\bibitem[HiPa13]{HiPa} M. Hindry, A. Pacheco, \emph{An analogue of the Brauer-Siegel theorem for abelian varieties in characteristic $p>0$}, Preprint, {https://sites.google.com/site/ amilcarpachecoresearch/publications} (2013) .
\bibitem[La83]{Lang} S. Lang,\emph{Fundamentals of diophantine geometry}, Springer-Verlag {\bf} (1983), chapter 6.
\bibitem[LaNe59]{LaNe} S. Lang, A. N\'eron, \emph{Rational points of abelian varieties over function fields}, Amer. J. Math. {\bf 81} No.1(1959): 95--118.
\bibitem[Liu94]{Liu} Q. Liu, \emph{Conducteur et discriminant minimal de courbes de genre 2}, Compositio Math. {\bf 94} (1994): 51--79.
\bibitem[LiSa00]{LiSa}Q. Liu, T. Saito, \emph{Inequality for conductor and differential of a curve over a local field}, J. Algebraic Geometry \textbf{9} (2000), 409-424.
\bibitem[Mi85]{Milne}J. S. Milne, \emph{Jacobian varieties}, in Arithmetic Geometry, eds. G. Cornell, J. H. Silverman, pp. 167-212, (1985).
\bibitem[Mi86]{mil}J. S. Milne, \emph{Arithmetic Duality Theorems}, Academic Press, (1986).
\bibitem[MB85]{mb}L. Moret-Bailly, \emph{Pinceaux de vari\'et\'es ab\'eliennes}, Ast\'erisque \textbf{129}, (1985).
\bibitem[Mu70]{mum} D. Mumford, \emph{Abelian Varieties}, Oxford University Press, (1970).
\bibitem[Ogg62]{Ogg1} A. P. Ogg, \emph{Cohomology of abelian varieties over function fields}, Ann. Math. {\bf 76} (1962): 185--212.
%\bibitem[Ogg67]{Ogg2} Ogg, A.P. \emph{Elliptic curves and wild ramification}, Amer. J. Math. {\bf 89} (1967): 1--21.
\bibitem[OoTa70]{taoo} F. Oort, J. Tate, \emph{Group schemes of prime order}, Ann. Sci. ENS Paris \textbf{3} (1970), 1-21.
\bibitem[Pa05]{Pach} A. Pacheco, \emph{On the rank of abelian varieties over function fields}, Manuscr. Math. {\bf 118} (2005): 361--381.
\bibitem[Pa12]{Paz} F. Pazuki,  \emph{Theta height and Faltings height}, Bull. SMF {\bf 140} No.1 (2012): 19--50.
%\bibitem[Ra85]{ra}M. Raynaud, \emph{Hauteurs et isog\'enies} in Ast\'erisque \textbf{127} (1985), 199-234.
\bibitem[Re10]{Rem} G. R\'emond, \emph{Nombre de points rationnels des courbes}, Proc. Lond. Math. Soc. {\bf 101} No.3 (2010): 759--794.
\bibitem[Sa66]{Sam} P. Samuel, \emph{Compl\'ement \`a un article de Hand Grauert sur la conjecture de Mordell}, IHES Publ. Math.{\bf 29} (1966): 55--62.
\bibitem[Se56]{serremex}J.-P. Serre, \emph{Sur la topologie des vari\'et\'es alg\'ebriques en caract\'eristique $p$} in ÷÷Symp. Top. Alg. Mexico City 1956'', pp. 24-53.
\bibitem[Se69]{Serre}J.-P. Serre, \emph{Facteurs locaux des fonctions z\^etas des vari\'et\'es alg\'ebriques (d\'efinitions et conjectures)}, in S\'em. Delange-Pisot-Poitou, (1969/70), no. 19.
\bibitem[Se79]{serrelf}J.-P. Serre, \emph{Local Fields}, Springer-Verlag, (1979).
\bibitem[SeTa68]{seta}J.-P. Serre, J. Tate, \emph{Good reduction of abelian varieties}, Annals of Math. \textbf{88} (1968), 492-517.
\bibitem[Sz81]{Szp} L. Szpiro, \emph{Propri\'et\'es num\'eriques du faisceau dualisant relatif}, Ast\'erisque {\bf 86} (1981): 44--78.
\bibitem[Ul91]{ul} D. L. Ulmer, \emph{$p$-descent in characteristic $p$}, Duke Math. J. \textbf{62} (1991), 237-265.
\bibitem[Vo91]{Vol} J. F. Voloch, \emph{On the conjectures of Mordell and Lang in positive characteristics}, Invent. Math. {\bf 104} No.3 (1991): 643--646.
\end{thebibliography}
\end{document}